\documentclass[12pt]{amsart}
\usepackage{a4}
\usepackage{amsthm, amsfonts, amssymb,latexsym}
\usepackage{enumerate, color}
\usepackage[english]{babel}
\usepackage[latin1]{inputenc}

\newtheorem{theorem}{Theorem}[section]

\newtheorem{prop}{Proposition}[section]

\newcommand{\F}{\mathbb{F}}
\newcommand{\Z}{\mathbb{Z}}
\newcommand{\la}{\lambda}

\def\Sym{{\rm Sym}}

\numberwithin{equation}{section}

\newcommand{\beq}[1]{\begin{equation}\label{#1}}
\newcommand{\eeq}{\end{equation}}

\title[Packing Sets]{Packing Sets} 
\author[O. Roche-Newton, I. D. Shkredov and A. Winterhof]{Oliver Roche-Newton, Ilya D. Shkredov and Arne Winterhof}
\address{O. Roche-Newton: Institute for Financial Mathematics and Applied Number Theory, Johannes Kepler Universit\"{a}t, Altenberger Str.\ 69,  Linz, Austria}
\email{o.rochenewton@gmail.com }
\address{I. D. Shkredov: Division of Number Theory, Steklov Mathematical Institute, ul.\ Gubkina 8, Moscow, 119991, Russia
and IITP, Bolshoy Karet\-ny Per. 19, Moscow, 127994,
and MIPT, Institutskii per. 9, Dolgoprudnii, 141701}
\email{ilya.shkredov@gmail.com}
\address{A. Winterhof: Johann Radon Institute for Computational and Applied Mathematics, Austrian Academy of Sciences, Altenberger Str.\ 69, Linz, Austria}
\email{arne.winterhof@oeaw.ac.at}


\begin{document}

\begin{abstract} 
For a given subset $A\subseteq G$ of a finite abelian group $(G,\circ)$, we study the problem of finding a large packing set $B$ for $A$, that is, a set $B \subseteq G$ such that $|A\circ B|=|A||B|$. Rusza's covering lemma and the trivial bound imply the existence of such a $B$
of size $|G|/|A|^2\le |G|/|A\circ A^{-1}|\le |B|\le |G|/|A|$. We show that these bounds are in general optimal and essentially any $\nu(A)$ in the interval $[|G|/|A|^2,|G|/|A|]$ can appear for some $|A|$.

The case that $G$ is the multiplicative group of the finite field $\F_p$ of prime order $p$ and $A=\{1,2,\ldots,\lambda\}$ for some positive integer $\lambda$ is particularly interesting in view of the construction of 
limited-magnitude error correcting codes. Here we construct a packing set $B$
of size $|B|\gg p (\lambda \log p)^{-1}$ for any $\lambda \le 0.9 p^{1/2}$. This result is optimal up to the logarithmic factor. 
\end{abstract} 
\maketitle

\section{Introduction}

Given two subsets $A$ and $B$ of a finite abelian group $(G,\circ)$ with unit $1$, the {\em product set} of $A$ and $B$ is defined as
$$A\circ B:=\{a\circ b:a \in A, b \in B \}.$$
We consider the cardinality of this product set, especially those sets for which the product set is of maximal size. 
A simple observation is that the trivial bound 
$$|A\circ B|\leq \min \{ |A||B|, |G| \}$$ 
holds for any $A,B \subseteq G$.

In this paper, we seek to answer the following question: given $\emptyset \not= A \subseteq G$, what is the size of the largest set $B \subseteq  G$ such that $|A\circ B|=|A||B|$?
We call any $B$ with $|A\circ B|=|A||B|$ an {\em $A$-packing set} and denote by $\nu(A)$ the maximal size of an $A$-packing set:
$$\nu(A):=\max\{|B|: B\subseteq G, |A\circ B|=|A||B|\}.$$

Suppose that we have such a set $B$. Since $|A||B|=|A\circ B|\leq |G|$, it must be the case that $|B| \leq |G|/|A|$ and thus
$$\nu(A)\le \left\lfloor \frac{|G|}{|A|}\right\rfloor.$$ 
For some interesting sets $A$, it can be easily established that $ \nu(A) $ is close to $|G|/|A|$. 
For example, if $A\subseteq G$ is a subgroup with distinct cosets $x_1\circ A,x_2\circ A,\dots,x_k\circ A$ where $k=|G|/|A|$, 
we can take $B=\{x_1,x_2,\dots,x_k\}$ and then $|A\circ B|=|A||B|=|G|$. Thus $\nu(A)=|B|=|G|/|A|$.
Conversely, if $A=\{x_1,x_2,\dots,x_k\}$ with elements in different cosets of a subgroup $B$ of order $|G|/k$, $B$ is an $A$-packing set.

The case that $G$ is the multiplicative group $\F_p^*$ of the finite field $\F_p$ of $p$ elements is particularly interesting in view of applications. More precisely, if $p$ is prime and $A=\{1,2,\ldots,\lambda\}$ for some positive integer $\lambda$, the authors in \cite{klboel,kllunaya,klluya} used an 
$A$-packing set $B$
to construct codes that correct single limited-magnitude errors. For more details see also \cite[Section 6.2.2]{niwi}. We denote
$$\nu(\lambda)=\nu(\{1,2,\ldots,\lambda\}).$$

Rusza's Covering Lemma, see \cite[Lemma 2.14]{tawu}, guarantees for any $A\subseteq G$ the existence 
of $B\subseteq G$ with $|A\circ B|=|A||B|$ and $G\subseteq A\circ A^{-1}\circ B$ and we get 
immediately  
\begin{equation}\label{bound1} \nu(A)\ge \left\lceil \frac{|G|}{|A\circ A^{-1}|}\right\rceil
\end{equation}
and since $|A\circ A^{-1}|\le |A|^2$
\begin{equation}\label{bound2} \nu(A)\ge \left\lceil\frac{|G|}{|A|^2}\right\rceil.
\end{equation}
For the convenience of the reader we will give a very short proof of $(\ref{bound1})$ in Section~\ref{sec2}.

In the above result, $A\circ A^{-1}$ denotes the set $\{a\circ b^{-1}:a,b \in A\}$, which we call the \textit{ratio set of $A$}. Note that the bound \eqref{bound1} is tight, up to multiplicative 
constants\footnote{Here and throughout the paper, the notation $X \ll Y$ and $Y \gg X$ indicates that there exists an absolute constant $c>0$ such that $X \leq cY$. 
If both $X \ll Y$ and $Y \ll X$, we write $X \approx Y$.}, in the case when the ratio set satisfies the bound $|A\circ A^{-1}| \ll |A|$. This generalises the result given by the simple construction above when $A$ is a multiplicative subgroup to the broader class of sets with small ratio set.

In fact, the weaker bound \eqref{bound2} is also optimal up to multiplicative constants in general, as the construction described in Section~\ref{sec2} shows. This construction can be modified to see that essentially any integer value $\nu(A)$ in the interval $[|G|/|A|^2,|G|/|A|]$ can be attained.


Section~\ref{sec4} deals with the special case when $G=\F_p^*$ with a prime $p$ and $A=\{1,2,\ldots,\lambda\}$. In this case we use the standard notation $AB$ for the product set, rather than $A \circ B$ as above. Since $|AA^{-1}|\gg \min\{\lambda^2,p\}$,
$(\ref{bound1})$ is only of limited power in this case. However, we give a simple construction which proves that\footnote{We denote by $\log$ the natural logarithm.} 
$$\nu(\lambda)\gg \frac{p}{\lambda \log p}$$
under the condition that $\lambda \leq 0.9p^{1/2}$.


Section~\ref{sec:symm} contains a result on the group of symmetries $\Sym(B)=\{x\in G : x\circ B=B\}$ of any  $A$-packing set $B$ of maximal size.

 Finally in Section \ref{sec:cov}, we briefly discuss the related problem of finding a small {\em $A$-covering set} $B$, that is, a set $B \subseteq  G$ such that $A \circ B=G$. 

\section{Proof of $(\ref{bound1})$ and proof of the optimality of $(\ref{bound2})$}\label{sec2}

\begin{proof}[Proof of $(\ref{bound1})$]

Let $B\subseteq G$ be any set with $\nu(A) =
|B|$.
Then, by the maximality of $B$, for each $x\in G$ we have $(A\circ x) \cap (A\circ B)\neq
\emptyset$, that is, $G \subseteq
A^{-1}\circ A\circ B$ and hence $|G| \le |A^{-1}\circ A\circ B| \le |A\circ A^{-1}| |B|$.
Thus $|B| \ge |G|/|A\circ A^{-1}|$.
\end{proof}

The following construction shows that $(\ref{bound2})$ is (up to a multiplicative constant) optimal.

Let $H=\{g,g^2,\dots,g^k\} \subseteq G$ be any cyclic subgroup of $G$ with $|H|=k\ge 2$. Let $d=\lceil \sqrt{k} \rceil\ge 2$ and define
$$A_1=\{g,g^2,\dots,g^d\},\,\,\,\,\,\,\,\, A_2=\{g^d,g^{2d},\dots,q^{(d-1)d},g^{d^2}\}.$$
Define $A=A_1 \cup A_2$. Note that $|A|<2d$ and that $A\circ A^{-1}=H$.

Now suppose that $|A\circ B|=|A||B|$ for some $B \subseteq G$. This is true if and only if there are no non-trivial solutions to the equation
$$a_1\circ b_1=a_2\circ b_2 ,\,\,\,\,\,\,\, (a_1,a_2,b_1,b_2) \in A \times A \times B\times B,$$
which happens if and only if
$$(A\circ A^{-1}) \cap (B\circ B^{-1} )= \{1\}.$$

We want to show that $B$ cannot be too large. Since $A\circ A^{-1}=H$, it must be the case that $(B\circ B^{-1} ) \cap H = \{1\}$. But then $B$ cannot contain more than one element from each coset of $H$. Indeed, if $b_1,b_2 \in B$ with $b_1=x\circ h_1$ and $b_2=x\circ h_2$ and with $h_1,h_2 \in H$ distinct, it follows that
$$b_1\circ b_2^{-1}=h_1\circ h_2^{-1} \in H \setminus \{1\} = A\circ A^{-1} \setminus \{1\}.$$
Therefore
$$|B| \leq \frac{|G|}{k} < \frac{|G|}{(d-1)^2} \leq \frac{16|G|}{|A|^2}.$$
This shows that $\nu(A) \ll |G|/|A|^2$. Furthermore, one can modify this construction by adding more elements from $H$ to the set $A$ in order to obtain, for any $0\leq \alpha \leq 1$, a set $A'$ with $|A'\circ A'^{-1}|\approx |A'|^{1+\alpha}$ and with $\nu(A') \ll |G|/|A'\circ A'^{-1}|$. This gives a broader class of sets for which the bound \eqref{bound1} is tight up to multiplicative constants.

\section{The case when $G=\F_p^*$ and $A=\{1,2,\dots, \lambda\}$}
\label{sec4}

In this Section, we consider the case of the multiplicative group $\F_p^*$ of a finite prime field and fix $A$ to be the interval $A=\{1,2, \dots , \lambda \} \subseteq \mathbb F_p^*$. Recalling the notation from the introduction, we seek lower bounds for $\nu(\lambda)$.
Inequality $(\ref{bound1})$ does not immediately give a strong result because of the following proposition.
\begin{prop}\label{prop} For $A=\{1,2,\ldots,\lambda\}\subseteq \F_p^*$ we have
$|AA^{-1}| \gg \min \{\lambda^2,p\}$.
\end{prop}
Proof. For the set $A_\Z=\{1,2,\ldots,\lambda\}$ of integers we have
$$A_\Z  A_\Z^{-1}=\left\{ab^{-1} : a,b\in A_\Z, \gcd(a,b)=1\right\}$$
and thus
$$|A_\Z A_\Z^{-1}|=\varphi(1)+2(\varphi(2)+\varphi(3)+\ldots+\varphi(\lambda))=\frac{6}{\pi^2}\lambda^2+O(\lambda \log \lambda)$$
by \cite[Theorem 330]{hawr}, where $\varphi$ is Euler's totient function.
If $\lambda< p^{1/2}$ and $1\le a_1,b_1,a_2,b_2\le \lambda$, then the congruence
$a_1b_1^{-1}\equiv a_2b_2^{-1}\bmod p$ is equivalent to the integer equation
$a_1/b_1=a_2/b_2$. Hence, the number of different elements of $AA^{-1}$ is the same as of $A_\Z/A_\Z$.
If $\lambda\ge p^{1/2}$, $A$ contains the subset $A'=\{0,1,\ldots,\lfloor p^{1/2} \rfloor\}$ and thus 
$|AA^{-1}|\ge |A'A'^{-1}|\gg p$.\hfill $\Box$\\

\textit{Remark.} For $\lambda\ge p^{1/2}\log^{1+\varepsilon}p$ we have $|AA^{-1}|=(1+o(1))p$, see \cite{ga}. 
This result was later extended to all $\lambda$ with $p^{1/2}=o(\lambda)$, see \cite[Theorem 1.7]{gaka}.
Also in \cite{ga}, it is mentioned that $AA^{-1}=\mathbb F_p^*$ if and only if $\lambda \geq \frac{p+1}{2}$.\\

With Proposition~\ref{prop} in mind, $(\ref{bound1})$ implies that $\nu(\lambda) \gg p/\lambda^2$. An explicit construction of such a set $B$ was given in 
\cite[Section 6.2.2]{niwi}.

In fact, we can provide a simple construction of a set $B$ which is almost as large as possible with the property that $|AB|=|A||B|$. Identify $\mathbb F_p$ with the set of integers $\{1,2,\dots,p \}$ in the obvious way and define
$$B:= \left \{ x \in \mathbb F_p : \lambda < x \leq \frac{p}{\lambda}, x \text{ is prime}\right \}.$$
This set has the property that $|AB|=|A||B|$. Indeed, suppose for a contradiction that we have a non-trivial solution to the equation
$$ab=a'b' ,\,\,\,\,\,\,\,(a,a',b,b') \in A \times A \times B \times B.$$
Since $A$ and $B$ are both contained in sufficiently small intervals, there are no wraparound issues, and so we must have a non-trivial solution to the equation
\begin{equation}
ab=a'b' ,\,\,\,\,\,\,\,(a,a',b,b') \in A_{\mathbb Z} \times A_{\mathbb Z} \times B_{\mathbb Z} \times B_{\mathbb Z},
\label{primes}
\end{equation}
where
$$A_{\mathbb Z}=\{1,2,\dots,\lambda \} \subseteq \mathbb Z,\,\,\,\,\,\, B_{\mathbb Z}=\{ x \in \mathbb Z : \lambda < x \leq \frac{p}{\lambda}, x \text{ is prime}\}.$$
However, unique prime factorisation of the integers implies that the only solutions to \eqref{primes} are trivial.

Furthermore, by the Prime Number Theorem,
$$|B| \gg \frac{p/ \lambda}{\log (p/\lambda)}- \frac{\lambda}{\log \lambda}.$$
In particular, if $ \lambda \le 0.9\sqrt{p}$, then we have $|B| \gg \frac{p}{\lambda \log p}$. We summarise this in the following statement:
\begin{theorem} \label{thm:interval}
Let $A=\{1,2 \dots, \lambda \} \subset \mathbb F_p^*$ with $ \lambda\le 0.9\sqrt{p}$. Then 
$$\nu(A) \gg \frac{p}{\lambda \log p}.$$
\end{theorem}

\begin{enumerate}
          \item Using explicit versions of the Prime Number Theorem, see \cite{rosc}, 
          $$c_1\frac{x}{\log x}\le \pi(x)\le c_2\frac{x}{\log x}\quad \mbox{if }x\ge x_0$$
          we can explicitly calculate the implied constant in Theorem~\ref{thm:interval}. 
          \item The same approach applies to any residue class ring $\Z_n$ with composite $n$.
          \item We may also take the larger packing set of {\em rough numbers} 
          $$B=\{x\in \F_p : \lambda < x\le \frac{p}{\lambda}, x \text{ is not divisible by a prime $\le \lambda$}\}.$$
          We have\footnote{We write $f(x)\sim g(x)$ if $\lim\limits_{x\rightarrow \infty} \frac{f(x)}{g(x)}=1$.} 
          $$|B| \sim \frac{p}{\lambda \log \lambda}\omega(u),$$
          where $\omega$ is {\em Buchstab's function} and $u=\frac{\log(p/\lambda)}{\log \lambda},$ see \cite{bu} or \cite[Paragraph IV.32]{samicr}.
          In particular, if $p^{1/3}\le \lambda\le p^{1/2}$, we have $1\le u\le 2$ and $\omega(u)=\frac{1}{u}=\frac{\log \lambda}{\log(p/\lambda)}$.
          However, for $u\rightarrow \infty$, the Buchstab function $\omega(u)$ converges to $e^{-\gamma}$, where $\gamma$ is the {\em Euler-Mascheroni constant}, see \cite{br}. 
          In particular, if\footnote{$f(x)=o(g(x))$ means $\lim\limits_{x\rightarrow\infty} \frac{f(x)}{g(x)}=0$.} $\lambda=e^{o(\log p)}$ is subexponential, we get $|B|\gg \frac{p}{\lambda \log \lambda}$ and so $\nu(\lambda)\gg \frac{p}{\lambda \log \lambda}$.
         \end{enumerate}

     \section{Symmetries} \label{sec:symm}

Let $A\subseteq G$ be a set and $B$ be an $A$-packing set.      
In this section we obtain a general result  about symmetries of our set of translations $B$ (this is in spirit of paper \cite{SSY}).  
Surprisingly, the set of symmetries of this extremal set $B$ does not grow  after taking the ratio $B\circ B^{-1}$.

Consider an arbitrary set  $T\subseteq G$.
Denote by $\Sym (T)$ the group of symmetries of $T$ that is
$$
\Sym (T) = \{ x \in G ~:~ x \circ T = T \} \,.
$$
Notice that $1\in \Sym (T)$, $\Sym(T) = \Sym^{-1} (T)$ and 
$\Sym(T) \subseteq T\circ T^{-1}$.

\begin{prop}
	Let $A\subseteq G$ be a set and let $B$ be an $A$-packing set of maximal size. 
	Then
	$$\Sym (B) = \Sym(B\circ B^{-1}) \,.$$
	Further  
	$$
	\left( \Sym (A\circ A^{-1})  \right ) \cap (A\circ A^{-1}) = (\Sym (A\circ A^{-1}) \setminus \Sym (B) ) \bigsqcup \{1\} \,. 
        \footnote{We denote by $A\sqcup B$ the union of two disjoint sets.}
	$$
\end{prop}
\begin{proof}
	The inclusion $\Sym (B) \subseteq \Sym(B\circ B^{-1})$ is trivial.
	Suppose that there is an element $x\in \Sym(B\circ B^{-1})$ but $x\notin \Sym (B)$.
	It follows that there is $b\in B$ such that $b\circ x \notin B$. 
	As in the proof of $(\ref{bound1})$ in Section~\ref{sec2} we get $B \circ A\circ A^{-1} \supseteq G$ and see that 
	$b \circ x = b' \circ  a_1 \circ a^{-1}_2$ for some $a_1,a_2\in A$ and $b'\in B$. 
	Hence because  $x\in \Sym(B\circ B^{-1})$, we get
	$$
	\tilde{b} \circ (\tilde{b}')^{-1} = b\circ x \circ (b')^{-1} = a_1 \circ a^{-1}_2
	$$
	for some $\tilde{b}, \tilde{b}'\in B$. 
	But $|A\circ B| = |A||B|$ and thus $a_1=a_2$, $\tilde{b} = \tilde{b}'$.
	It gives us $b\circ x = b' \in B$ and this is a contradiction.

	Taking $x\in \Sym (A\circ A^{-1}) \setminus \Sym (B)$ and repeating the previous arguments, we obtain 
	$$
	b \circ (b')^{-1} = a_1 \circ (x \circ a_2)^{-1} = \tilde{a}_1 \circ (\tilde{a}_2)^{-1} 
	$$
	and hence $b=b'$, $\tilde{a}_1 = \tilde{a}_2$. 
	Thus $x=a_1 \circ a^{-1}_2 \in A \circ A^{-1}$ and we get
	$$
	\Sym (A\circ A^{-1}) \setminus \Sym (B) \subseteq A \circ A^{-1} \,.
	$$
	But $\Sym (B) \subseteq B\circ B^{-1}$ and $(B\circ B^{-1}) \cap (A\circ A^{-1}) = \{ 1\}$ thus 
	$\Sym(B) \cap (A\circ A^{-1}) = \{ 1\}$. 
	This completes the proof.  
\end{proof}

\bigskip

If $B$ is any $A$-packing set of maximal size, then
the appearance of the set $\Sym (B)$ in our problem of computing $\nu (A)$ is natural 
in view 
of a trivial equality $\nu (A \circ \Sym(B)) = \nu (A) = |B|$.

         \section{Covering sets} \label{sec:cov}

Given $A \subseteq G$, we say that $B \subseteq G$ is an {\em $A$-covering set} if $A \circ B =G$. The {\em covering number} of $A$, denoted $cov(A)$, is the size of the smallest $A$-covering set. There is a natural connection between covering and packing problems, and likewise with the problems of determining the values of $cov(A)$ and $\nu(A)$. In particular, it follows from Ruzsa's Covering Lemma that 
$$cov(A \circ A^{-1}) \leq \nu(A).$$

The problem of determining $cov(A)$ in the case $G=\mathbb F_p^*$ was studied in \cite{chshwi,klsh1,klsh2}, where $A=\{1,2,\ldots,\lambda\}$. A more general study of the problem can be found in \cite{BJR}; see Section 3 therein for background on this problem in the finite setting. In particular, it is proved in \cite[Corollary 3.2]{BJR} that for any finite group $G$ and $A \subset G$
\begin{equation}
\frac{|G|}{|A|} \leq cov(A) \leq \frac{|G|}{|A|}(\log|A| +1). 
\label{cov}
\end{equation}
By contrast with \eqref{cov}, we showed in Section 2 of this paper that $\nu(A)$ can essentially take any value in between $|G|/|A|^2$ and $|G|/|A|$. It is interesting to note that the size of $cov(A)$ is much more restricted than that of $\nu(A)$. 

In the special case $G=\F_p^*$ and $A=\{1,2,\ldots,\lambda\}$ we have the improvement $cov(A)< 2p/\lambda$ by \cite[Theorem 2]{chshwi}.  However, an interesting observation is that if we instead take $A$ to be the middle third interval then the log factor is needed and $cov(A) \approx \log |A|$.
In particular, this gives us a constructive example (as opposed to random choice, see, say, \cite{BJR}) of a set such that  upper bound in $(\ref{cov})$ is sharp.

         \begin{prop}
	 For a prime $p>3$ put
$$
	A= \{ x\in \F_p^* ~:~ x\in [p/3, 2p/3] \} \,.
$$
	Then we have 
    $$\frac{\log(p-1)}{\log(3)} \le cov(A) < 3(\log(p)+1).$$ 
\end{prop}
\begin{proof}
	Put $T= \{ x\in \F_p^* ~:~  x \not\in[p/3,2p/3] \}$. For $\lambda \in \{1,\ldots,p-1\}$ let inv$(\lambda)\in \{1,\ldots,p-1\}$ be the unique integer
    with ${\rm inv}(\lambda)\lambda\equiv 1\bmod p$.
	By the simultaneous version of the Dirichlet Approximation Theorem, see \cite{sc}, for any integer $1\le k<\log (p-1)/\log(3)$ and $\la_1,\dots, \la_k \in \{1,\ldots,p-1\}$  there is an integer $1\le n<p$ 
    and integers $a_1,\ldots,a_k$ such that 
    $$|{\rm inv}(\lambda_i)n/p-a_i|\le 1/(p-1)^{1/k}< 1/3$$ 
    for $i=1,\ldots,k$.
  	In other words, for any  $\la_1,\dots, \la_k\in \F_p^*$ with $1\le k< (p-1)/\log(3)$ there is 
	$n\in \la_1 T \cap \dots \cap \la_k T$.
	Putting $B = \{ \la_1,\dots, \la_k  \}$, we see  
	that $n \notin  AB$ and hence  $ AB \neq \F^*_p$ for any $B$ with $1\le |B| <\log(p-1)/\log(3)$. 
	By the definition this means that  $cov (A) \ge\log (p-1)/\log(3)$. The upper bound follows from $(\ref{cov})$. 
\end{proof}

            

\section*{Acknowledgements}

The first and third authors are supported by the Austrian Science Fund FWF Projects F5509 and F5511-N26, respectively, 
which are  part of the Special Research Program ``Quasi-Monte Carlo Methods: Theory and Applications". 
We are grateful to Antal Balog, Brandon Hanson, George Shakan and Igor Shparlinski for helpful conversations and insights.

\end{document}